\documentclass[12pt,a4paper,reqno]{amsart} \usepackage[left=2.7cm,right=2.7cm,top=3cm,bottom=3cm,a4paper]{geometry}
\usepackage{amsmath,amscd,amssymb,latexsym}
\usepackage{enumerate}

\numberwithin{equation}{section} 

\newenvironment{psmallmatrix}
  {\left(\begin{smallmatrix}}
  {\end{smallmatrix}\right)}
 
\usepackage{color}
\definecolor{forestgreen}{rgb}{0.0, 0.27, 0.13}
\definecolor{ultramarine}{rgb}{0.07, 0.04, 0.56}
\usepackage[pagebackref=true,colorlinks=true,linkcolor=ultramarine,citecolor=ultramarine,urlcolor=forestgreen]{hyperref}

\newtheorem{thm}{Theorem}[section]
\newtheorem{lem}[thm]{Lemma}
\newtheorem{prop}[thm]{Proposition}
\newtheorem{defn}[thm]{Definition}

\newtheorem{conj}[thm]{Conjecture}

\newtheorem{rmk}[thm]{Remark}

\newenvironment{pf}{{\noindent \bf Proof:\ }}{\hfill
$\Box$ \bigskip}

\newcommand{\QQ}{\mathbb{Q}}
\newcommand{\ZZ}{\mathbb{Z}}

\newcommand{\NN}{\mathbb{N}}
\newcommand{\CC}{\mathbb{C}}
\newcommand{\FF}{\mathbb{F}}

\def\ord{\mathrm{ord}}
\def\cA{{\mathcal A}}
\def\cC{{\mathcal C}}
\def\cG{{\mathcal G}}
\def\fp{\mathfrak{p}}

\begin{document}
\setlength{\unitlength}{1mm}

\title[$p$-adic properties for Taylor coefficients of modular forms on $\Gamma_1(4)$] 
{$p$-adic properties for Taylor coefficients of half-integral weight modular forms on $\Gamma_1(4)$}

\author{Jigu Kim and Yoonjin Lee}

\date{}


\maketitle

\noindent{\small {\bf Abstract.} 
For a prime $p\equiv 3$ $(\text{mod }4)$ and $m\ge 2$, Romik raised a question about whether the Taylor coefficients around $\sqrt{-1}$ of the classical Jacobi theta function $\theta_3$ eventually vanish modulo $p^m$.  This question can be extended to a class of modular forms of half-integral weight on $\Gamma_1(4)$ and CM points; in this paper, we prove an affirmative answer to it for primes $p\ge5$. Our result is also a generalization of the results of Larson and Smith for modular forms of integral weight on $\mathrm{SL}_2(\ZZ)$.}

\section{Introduction and results}
Let $\ZZ$ be the ring of integers, $k\in \ZZ$ or $\frac{1}{2}+\ZZ$, and $p\ge 5$ be a prime. Let $R\subset \CC$ be a commutative ring. Let $M_{k}(R,\Gamma_1(4))$ be the space of modular forms of integral or half-integral weight on $\Gamma_1(4)$ whose Fourier coefficients at the cusp $\infty$ belong to $R$.

For $z=x+iy$ in the upper-half plane $\mathbb{H}$, let $q=e^{2\pi i z}$ and $f(z)=\sum a_nq^n\in M_k(R,\Gamma_1(4))$.
Let $D$ be the derivative with respect to $2\pi i z$; that is,
$$Df:=\frac{1}{2\pi i}\frac{d}{d z}f=q\frac{d}{dq}f.$$
We denote by $\partial$ the {\it 
Shimura-Maass differential operator}, 
which is defined by
$$\partial f:=\partial_k f:=Df-\frac{k}{4\pi y}f.$$
For a given point $\tau_0=x_0+iy_0\in \mathbb{H}$,
we note that the transformation 
$z\mapsto w={(z-\tau_0)}/{(z-\overline{\tau}_0)}$
maps $\mathbb{H}$ to the open unit disc around $0$ and sends $\tau_0$ to the origin.
The Taylor expansion of $f(z)=f\big({(\tau_0-\overline{\tau}_0w)}/{(1-w)}\big)$ around $w=0$ gives rise to
\begin{equation}\label{TS-formula}
(1-w)^{-k}f\Big(\frac{\tau_0-\overline{\tau}_0w}{1-w}\Big)
=\sum_{n=0}^\infty \partial^n f(\tau_0)
\frac{(-4\pi y_0 w)^n}{n!},
\end{equation}
where $|w|<1$ and for $k\in \frac{1}{2}+\ZZ$ we take the branch of the square root having argument in $(-\pi/2,\pi/2]$. (As mentioned in \cite{VW14}, we note that there is the sign error for the formula in \cite[Proposition 17]{Zag08} and \cite[(1.1)]{LS14}.)
In order to obtain an algebraic number from the Taylor coefficient, we further suppose that $\tau_0$ is a CM point and divide $\partial^n f(\tau_0)$ by a power of  some transcendental factor. 

Let $K$ be an imaginary quadratic field of discriminant $d$ and $\tau_0$ be a CM point in $K$.
Let $h(d)$ be the class number of $K$, $w_d$ the number of roots of unity in the ring of integers $\mathcal{O}_K$, and $\chi_d$ the quadratic character associated to $K$.
We define $\Omega_K\in\CC$ as
$$\Omega_K:=
\frac{1}{\sqrt{2\pi |d|}} 
\left(\prod_{k=1}^{d} \gamma\left(
\frac{k}{|d|}\right)^{\chi_d(k)}\right)^{\frac{w_d}{4h(d)}},$$
where $\gamma$ is the Gamma function. 
We note that $\Omega_K$ is related with the {\it Chowla-Selberg formula}. 
It is also well-known that for $f\in M_{k}(\overline{\QQ},\Gamma')$ with any congruence subgroup 
$\Gamma'\subset \mathrm{SL}_2(\ZZ)$ such that $\Gamma'\subset \Gamma_1(4)$ if $k\in\frac{1}{2}+\ZZ$, we have
$$\partial^n f(\tau_0)\in \overline{\QQ}\cdot \Omega_K^{2n+k}$$
(we refer to \cite[p. 86]{Zag08}). Then we can take a complex number $\Omega_{\tau_0}\in \overline{\QQ}\cdot \Omega_K$ 
such that for every $k\in\ZZ$ or $\frac{1}{2}+\ZZ$, every $f\in M_k({\ZZ},\Gamma_1(4))$ and all primes $p\ge 5$, we have 
\begin{equation}\label{eq-trascendental}
\min_{\substack{n\in\NN,\,k\in\ZZ \text{ or } \frac{1}{2}+\ZZ,\\f\in M_k({\ZZ},\Gamma_1(4))}}
\left\{\ord_p\left(
\frac{\partial^n f(\tau_0)}{\Omega_{\tau_0}^{2n+k}}\right)
\right\}\ge 0, 
\end{equation}
where the $p$-adic valuation $\ord_p$ is mentioned in \eqref{p-ord};
we note that $\Omega_{\tau_0}$ only depends on $\tau_0$.

Let $k'\in2\ZZ$ and $\Gamma(1):=\mathrm{SL}_2(\ZZ)$. Larson and Smith defined $\Omega_{\tau_0}^{'}$ for $M_{k'}(\ZZ,\Gamma(1))$ similarly and proved the following theorem for a CM point $\tau_0$ in $K$ and an inertial or ramified prime $p$ in $K$.
\begin{thm}\label{LS-thm}{\rm \cite[Theorem 1.3]{LS14}}
Let $p\ge 5$ be a prime and $k'\in2\ZZ$.
Let $f\in M_{k'}(\ZZ,\Gamma(1))$, and let $\tau_0$ be a CM point in $\QQ(\sqrt{d})$ of discriminant $d<0$. 
If $p$ satisfies $\big(\frac{d}{p}\big)
\in\{0,-1\}$, then we have
$$\frac{(\partial^n f)(\tau_0)}{\Omega_{\tau_0}^{'2n+k}}
\equiv 0 \quad (\text{mod } p^m)$$
for all $m\ge 2$ and $n\ge (m-1) p^2$.
\end{thm}

Recently, many authors studied the Taylor coefficients of the classical Jacobi theta
$$\theta_3(z):=\sum_{n=-\infty}^{\infty}e^{\pi i n^2 z}$$
around $\tau_0=i$ (we refer to \cite{GMR20,Rom20,Sch21,Wak20,WV20}). Romik defined $(d(n))_{n=0}^{\infty}$ to be the sequence such that 
\begin{equation}\label{Romik-seq}
(1-w)^{-1/2}\theta_3\left(
\frac{i+wi}{1-w}\right)
=\theta_3(i)\sum_{n=0}^{\infty}\frac{d(n)}{(2n)!}
\left(\frac{\gamma(1/4)^4}{8\pi^2 \sqrt{2}}w\right)^{2n},\quad |w|<1,
\end{equation}
and showed that $d(n)$'s are integers \cite[Theorems 1 and 2]{Rom20}. 
Comparing \eqref{TS-formula} with \eqref{Romik-seq},
we note that $\partial^n \theta_3(i)=0$ for odd $n$; this is because for $z=z(w)=(i+wi)/(1-w)$, we have that 
$\theta_3(z(-w))=\theta_3(-1/z)=\sqrt{z/i}\theta_3(z)
=((1+w)/(1-w))^{1/2}\theta_3(z(w))$.
Romik also made the following conjecture.
\begin{conj}\label{Rom-conj}{\rm \cite[Conjecture 13(b) and Open problem 2]{Rom20}}
Let $p$ be an odd prime. Then we have:
\begin{enumerate}[(a)]
\item\label{conj-a} If $p\equiv3$ $(\text{mod }4)$, then $d(n)\equiv 0$ $(\text{mod }p)$ for sufficiently large $n$. 
\item\label{conj-b} If $p\equiv1$ $(\text{mod }4)$, the sequence 
$\{d(n)\,\,(\text{mod }p)\}_{n=0}^{\infty}$ is periodic.
\item\label{conj-c} If $p\equiv3$ $(\text{mod }4)$ and $m\ge2$, then $d(n)\equiv 0$ $(\text{mod }p^m)$ for sufficiently large $n$. 
\item\label{conj-d} If $p\equiv1$ $(\text{mod }4)$ and $m\ge2$, the sequence 
$\{d(n)\,\,(\text{mod }p^m)\}_{n=0}^{\infty}$ is periodic.
\end{enumerate}
\end{conj}

Part \eqref{conj-a} is proved by Scherer \cite[Theorem 1]{Sch21}: Scherer showed that $d(n)\equiv 0$ $(\text{mod }p)$ for $p\equiv3$ $(\text{mod }4)$ and $n\ge (p^2+1)/2$.
Both parts \eqref{conj-b} and \eqref{conj-d} are shown by Guerzhoy et al. \cite{GMR20}. 
Guerzhoy et al. also generalized \eqref{conj-d} to a broader class: $f\in M_k(\overline{\ZZ},\Gamma_1(4N))$ ($k\in\ZZ$ or $\frac{1}{2}+\ZZ$, $N\in\NN$), CM points $\tau_0\in K$, and splitting primes $p$ in $\mathcal{O}_K$ (see \cite[Theorem 1.2]{GMR20}). 

In the following Theorem \ref{main-thm-0}, we generalize Theorem \ref{LS-thm} to the space of modular forms of half-integral weight on $\Gamma_1(4)$. Applying Theorem \ref{main-thm-0} to Conjecture \ref{Rom-conj}\eqref{conj-c} for $p\ge 5$, we prove that
Conjecture \ref{Rom-conj}\eqref{conj-c} 
holds except for $p=3$; this is stated in Theorem \ref{cor-thm-0}.

\begin{thm}\label{main-thm-0}
Let $p\ge 5$ be a prime, $m\in\NN$ and $2k\in\ZZ$. 
Let $f\in M_{k}(\ZZ,\Gamma_1(4))$, and let $\tau_0$ be a CM point in $\QQ(\sqrt{d})$ of discriminant $d<0$. 
If $p$ satisfies $\big(\frac{d}{p}\big)
\in\{0,-1\}$, then we have
$$\frac{(\partial^n f)(\tau_0)}{\Omega_{\tau_0}^{2n+k}}
\equiv 0 \quad (\text{mod } p^m),$$
where $m\ge 2$, $n\ge (m-1) p^2$, and $\Omega_{\tau_0}$ satisfies \eqref{eq-trascendental}.
\end{thm}

For $p\ge5$, Conjecture \ref{Rom-conj}\eqref{conj-c} follows from Theorem \ref{main-thm-0} by taking that
$f(z)=\sum_{n\in\ZZ}q^{n^2}\in M_{1/2}(\ZZ,\Gamma_1(4))$ and $\tau_0=i/2\in \QQ(\sqrt{-4})$.
\begin{thm}\label{cor-thm-0}
Let $p\ge5$ be a prime and $m\in\NN$.
If $p\equiv 3$ $(\text{mod }4)$, $m\ge 2$ and $n\ge \lceil(m-1) p^2/2\rceil$, then we have $d(n)\equiv 0$ $(\text{mod }p^m)$, where $d(n)$ is defined in \eqref{Romik-seq}.
\end{thm}
\smallskip

\section{The algebra of $\Gamma_1(4)$-quasimodular forms}\label{sec-2}

In this section we recall some standard facts about the algebra of modular forms of half-integral weight on $\Gamma_1(4)$ and their derivatives. 

To begin with, we recall definitions of modular forms of integral or half-integral weight. For a function $f:\mathbb{H}\to \CC$ of weight $k\in\ZZ$ and a matrix 
$\gamma={\begin{psmallmatrix}a&b\\c&d\end{psmallmatrix}}\in \mathrm{SL}_2(\ZZ)$, we define the slash operator as
$$(f|_k \gamma)(z):=(cz+d)^{-k}f(\gamma\cdot z), \qquad {\rm where} \; \gamma\cdot z:=\frac{az+b}{cz+d}.$$
In the case of weight $k\in\frac{1}{2}+\ZZ$, we further assume that a matrix
$\gamma={\begin{psmallmatrix}a&b\\c&d\end{psmallmatrix}}\in \Gamma_0(4)$ (i.e., $4\mid c$), and we define the slash operator as
$$(f|_k \gamma)(z):=
\left(\frac{c}{d}\right) \varepsilon_{d}^{2k} (\sqrt{cz+d})^{-2k} 
f(\gamma\cdot z),$$
where we take the branch of the square root having argument in $(-\pi/2,\pi/2]$, $(\tfrac{c}{d})$ is the extended Jacobi symbol (see \cite{Shi73} or \cite[p. 178]{Kob86}), and 
$$\varepsilon_d=
\left\{\begin{array}{ll}
1 &\text{if }d\equiv 1 \,\,(\text{mod }4),\\
i &\text{if }d\equiv 3 \,\,(\text{mod }4).
\end{array}\right.$$
For a level $N\in\NN$ ($4\mid N$ if $k\in\frac{1}{2}+\ZZ$), let $\chi$ be a Dirichlet character modulo $N$. 
A function $f:\mathbb{H}\to \CC$ is called a holomorphic modular form with Nebentypus $\chi$ of weight $k\in\ZZ$ or $\frac{1}{2}+\ZZ$ on $\Gamma_0(N)$ if it is holomorphic on $\mathbb{H}$ and at the cusps of $\Gamma_0(N)$, and if 
$$(f|_k \gamma)(z)=\chi(d)f(z)$$
for all $\gamma \in \Gamma_0(N)$.
For such $f$ whose Fourier coefficients belong to a ring $R$ ($\ZZ\subset R\subset \CC$), we write $f\in M_k(R,\Gamma_0(N),\chi)$, and we simply write $f\in M_k(R,\Gamma_0(N))$ if $\chi=\chi_{\text{triv}}$.
We note that $-I={\begin{psmallmatrix}-1&0\\0&-1\end{psmallmatrix}}\in \Gamma_0(N)$, $-I\cdot z=z$, $(f|_k (-I))(z)=-f(z)$ if $k\in1+2\ZZ$ and $(f|_k (-I))(z)=f(z)$ otherwise.
Therefore, using the fact that $M_k(R,\Gamma_1(4))=M_k(R,\Gamma_0(4))\oplus M_k(R,\Gamma_0(4),\chi_{-4})$,
it follows that
$$M_k(R,\Gamma_1(4))=\left\{
\begin{array}{ll}
M_k(R,\Gamma_0(4)) &\text{if }k\in\tfrac{1}{2}+\ZZ,\\
M_k(R,\Gamma_0(4),\chi_{-4}) &\text{if }k\in 1+2\ZZ,\\
M_k(R,\Gamma_0(4)) &\text{if }k\in 2\ZZ.
\end{array}\right.$$
Then we have that 
$$M_{\ast}(R,\Gamma_1(4)):=\bigoplus_{j=0}^{\infty}M_{j/2}(R,\Gamma_1(4))$$ 
is the graded $R$-algebra.

For $j\ge 1$ and $n\in\NN$, let $\sigma_j(n):=\sum_{0<d|n}d^j$. 
Let $\Theta$ and $F_2$ be the classical forms 
$$\Theta:=\sum_{n\in\ZZ}q^{n^2}\in M_{1/2}(\ZZ,\Gamma_1(4))
\quad \text{ and } \quad 
F_2:=\sum_{n \text{ odd }\ge1} \sigma_1(n)q^n
\in M_{2}(\ZZ,\Gamma_1(4)).$$
It is well-known that 
$$M_{\ast}(\CC,\Gamma_1(4))=\CC[\Theta,F_2].$$ 
Furthermore, the next lemma follows by the same proof as in \cite[p. 184]{Kob93}.

\begin{lem}\label{p-integral-coefficients}
Let $R$ be a ring such that $\ZZ\subset R \subset \CC$ and $R_{(6)}:=R[\frac{1}{2},\frac{1}{3}]$.
Let $k\in\ZZ$ or $\frac{1}{2}+\ZZ$ and $f\in M_k(R,\Gamma_1(4))$. Then there exist elements $c_{a,b}\in R_{(6)}$ such that  
$$f=\sum_{\substack{0\le a,b\in\ZZ, \\a/2+2b=k}}c_{a,b}\Theta^a F_2^{b}.$$
\end{lem}

For even $k\in\NN$, let $E_k$ be the classical Eisenstein series, defined as 
$$E_k=1-\frac{2k}{B_k}\sum_{n\ge1}\sigma_{k-1}(n)q^n,$$
where $B_k$ is the $k$-th Bernoulli number. 
We note that $E_k\in M_k(\QQ,\Gamma(1))$ for $k\ge 4$, but $E_2$ is not modular. We recall that $D$ is the derivative with respect to $2\pi i z$. The derivative of a modular form of integral or half-integral weight is no longer modular but quasimodular, which means that in the case of $\Gamma_1(4)$ it is an isobaric element of the ring $\CC[\Theta,F_2,E_2]$. The derivative $D$ preserves the ring $R_{(6)}[\Theta,F_2,E_2]$ since we have 
\begin{equation}\label{D-basis}
\left\{
\begin{array}{l} 
D\Theta=(\Theta E_2 -\Theta^5 +80\Theta F_2 )/24,\\
DF_2=(F_2 E_2+5\Theta^4 F_2-16F_2^2)/6,\\
DE_2
=(E_2^2-\Theta^8-224\Theta^4F_2-256F_2^2)/12.
\end{array}
\right.\end{equation}

To any $g\in\CC[\Theta,F_2,E_2]$, 
we attach a polynomial $\cG(g; X,Y,Z)$ such that 
$$g(z)=\cG(g;\Theta(z),F_2(z),E_2(z)).$$
We denote by $g_0$ the {\it modular part} of $g$, that is, $g_0(z):=\cG(g;\Theta(z),F_2(z),0)$. When $g=g_0$, we simply write $\cG(g;X,Y)\in \CC[X,Y]$ instead of $\cG(g;X,Y,Z)\in \CC[X,Y,Z]$.

For modular forms of even integral weight on $\Gamma(1)$, it is well-known that 
$$\bigoplus_{j=0}^{\infty} M_{2j}(\CC,\Gamma(1))
= \CC[E_4,E_6].$$
The derivative $D$ preserves the ring of $\Gamma(1)$-quasimodular forms $\CC[E_4,E_6,E_2]$ since 
$DE_2=(E_2^2-E_4)/12$, $DE_4=(E_2E_4-E_6)/3$ and $DE_6=(E_2E_6-E_4^2)/2$; so we define $G(g;X,Y,Z)$ such that 
$$g(z)=G(g;E_4(z),E_6(z),E_2(z)).$$

We have that for $f\in M_{2j}(\CC,\Gamma(1))$,
\begin{equation}\label{basis-basis-0}
\cG(f;X,Y)=G(f;X^8+224X^4Y+256Y^2,
X^{12}-528X^8Y-8448X^4Y^2+4096Y^3);
\end{equation}
this is because we have that
\begin{equation}\label{E4-F2}
\left\{
\begin{array}{l}
E_4=\Theta^8+224\Theta^4F_2+256F_2^2,\\
E_6= \Theta^{12}-528\Theta^8F_2
-8448\Theta^4F_2^2+4096F_2^3.
\end{array}\right.
\end{equation}

Now, let $p\ge 5$ be a prime, let $\ZZ_{(p)}$ be a local ring at $p$, and let $n\ge0$ be an integer. By Lemma \ref{p-integral-coefficients} and \eqref{D-basis}, if a $\Gamma_1(4)$-modular form $g$ has $p$-integral Fourier coefficients, then $\cG(D^{n}g;X,Y,Z)$ also has $p$-integral coefficients. For $m\in\NN$, we denote by $\overline{g}\in (\ZZ/p^m\ZZ)[[q]]$ (resp., $\overline{\cG}(D^ng;X,Y,Z)\in (\ZZ/p^m\ZZ)[X,Y,Z]$) the image obtained by reducing its Fourier coefficients (resp., coefficients) mod $p^m$ under the the canonical map $\ZZ_{(p)}\to \ZZ/p^m\ZZ$. 
The same property holds for a $\Gamma(1)$-modular form $g\in \ZZ_{(p)}[[q]]$; hence we similarly define $\overline{G}(D^{n}g;X,Y,Z)\in (\ZZ/p^m\ZZ)[X,Y,Z]$. 

Also, it is well-known that $E_{p-1}\in M_{p-1}(\ZZ_{(p)},\Gamma(1))$ and $E_{p+1}\in M_{p+1}(\ZZ_{(p)},\Gamma(1))$. Let $\cA_p$ and $A_p$ be defined by
$$\cA_p:=\cG(E_{p-1};X,Y),\qquad 
A_p:=G(E_{p-1};X,Y),$$
which are elements in $\ZZ_{(p)}[X,Y]$.
\smallskip

\section{Mod $p^m$ modular forms}\label{sec-3}

Many authors generalized the theories of Serre \cite{Ser73} and Katz \cite{Kat73} regarding $p$-adic congruences of modular forms of integral weight to those of half-integral weight (we refer to \cite{Kob86,Tup06}). In this section we study some analogous facts on modulo $p^m$ congruences of $\Gamma_1(4)$-modular forms of half-integral weight in terms of $(\ZZ/p^m\ZZ)[X,Y]$.

Let $k\in\ZZ$ or $\frac{1}{2}+\ZZ$ and $m\in\mathbb{N}$. 
Let $\Gamma'\subset \Gamma(1)$ be a congruence subgroup such that $\Gamma'\subset \Gamma_1(4)$ if $k\in\frac{1}{2}+\ZZ$. By $\widetilde{M}_k(\ZZ/p^m\ZZ,\Gamma')$ we denote the $\ZZ/p^m\ZZ$-module (in $(\ZZ/p^m\ZZ)[[q]]$) obtained from $M_k(\ZZ_{(p)},\Gamma')$ by reducing its Fourier coefficients mod $p^m$. The {\it mod $p^m$ filtration} of 
$\overline{f}\in \widetilde{M}_k(\ZZ/p^m\ZZ,\Gamma')$ is defined to be
$$w_{p^m}(\overline{f}):=\inf\{k' \mid \overline{f}=\overline{h} \,\text{ for some }\,\overline{h}\in\widetilde{M}_{k'}(\ZZ/p^m\ZZ,\Gamma')\},$$ 
where we have the convention that the modular form $0$ has weight $-\infty$.

The following theorem is an analogue of partial results by Serre and Katz (see \cite[Ch. X, Theorem 7.5]{Lan95} and \cite[Lemma 2.4]{LS14}), and it will be used to prove Lemma \ref{LS-prop-4.3}.

\begin{thm}\label{Kob-Lan} 
Let $p\ge5$ be a prime, $k\in\ZZ$ or $\frac{1}{2}+\ZZ$, and $m\in\NN$.
Let $f\in M_k(\ZZ_{(p)},\Gamma_1(4))$ such that 
$f\not\equiv 0$ $(\text{mod }p\ZZ_{(p)}[[q]])$. The filtration $w_{p^m}(\overline{f})$ is less than $k$ 
if and only if 
$\overline{\cA}_p(X,Y)^{p^{m-1}}$ divides $\overline{\cG}(f;X,Y)$ in $(\ZZ/p^m\ZZ)[X,Y]$. 
\end{thm}
\begin{pf}
$(\Leftarrow)$ It is trivial since $\overline{\cA}_p(\Theta,F_2)^{p^{m-1}}\equiv\overline{E}_{p-1}^{p^{m-1}}\equiv 1$ $(\text{mod }p^m\ZZ_{(p)}[[q]])$.\\ 
$(\Rightarrow)$ 
If $k\in\ZZ$, it directly follows from \cite[Corollary 4.4.2]{Kat73}. Note that if $g$, $h\in\ZZ_{(p)}[[q]]$ are not congruent to $0$ $(\text{mod }p\ZZ_{(p)}[[q]])$, then neither is $gh$ as $(\ZZ/p\ZZ)[[q]]$ is an integral domain. Now, we assume that $k\in \frac{1}{2}+\ZZ$, $f\in M_k(\ZZ_{(p)},\Gamma_1(4))$, $f\not\equiv 0$ $(\text{mod }p\ZZ_{(p)}[[q]])$ and $w_{p^m}(\overline{f})<k$. Then we have that $k+\tfrac{1}{2}\in \ZZ$, 
$\Theta f\in M_{k+\frac{1}{2}}(\ZZ_{(p)},\Gamma_1(4))$, $\Theta f\not\equiv 0$ $(\text{mod }p\ZZ_{(p)}[[q]])$ and $w_{p^m}(\overline{\Theta f})<k+\tfrac{1}{2}$; so, we have that in $(\ZZ/p^m\ZZ)[X,Y]$ 
$$\overline{\cA}_p(X,Y)^{p^{m-1}} \mid \overline{\cG}(\Theta f; X,Y).$$ 
We note that $\cG(\Theta f; X,Y)=X\cG(f; X,Y)$. We recall $A_p(X,Y):=G(E_{p-1};X,Y)\in\ZZ_{(p)}[X,Y]$, and we define $\widetilde{A}_p(X,Y)\in \FF_p[X,Y]$ to be its reduction mod $p$. Let $\overline{\FF}_p$ be the algebraic closure of $\FF_p$. It is well-known that over $\overline{\FF}_p$, the irreducible factors of $\widetilde{A}_p(X,Y)$ must be of the form 
\begin{equation}\label{Lang-text}
X,\,\,\, Y,\,\, \text{ or }\,\, X^3-\alpha Y^2\,\, \text{ with }\,\, \alpha\neq1
\end{equation} 
(see \cite[pp. 166-167]{Lan95}). Let $\langle X,p\rangle$ be the ideal generated by $X$ and $p$ in the ring $(\ZZ/p^m\ZZ)[X,Y]$. By \eqref{basis-basis-0} we have that in $\big((\ZZ/p^m\ZZ)[X,Y]\big)/\langle X,p\rangle$
$$\overline{\cA}_p(0,Y)+\langle X,p\rangle
=\widetilde{A}_p((16Y)^2,(16Y)^3)+\langle X,p\rangle
=\beta(16Y)^{(p-1)/2}+\langle X,p\rangle,$$
for some $\beta\in \FF_p$. 
By \eqref{Lang-text}, we have that $\beta$ is nonzero; so, we get $X\nmid \overline{\cA}_p(X,Y)$ in $(\ZZ/p^m\ZZ)[X,Y]$. 
Therefore, $\overline{\cA}_p(X,Y)^{p^{m-1}}$ divides $\overline{\cG}(f;X,Y)$ in $(\ZZ/p^m\ZZ)[X,Y]$. 
\end{pf}

\begin{rmk}{\rm 
(a) In the proof of Theorem \ref{Kob-Lan} we show the property that $X\nmid \overline{\cA}_p(X,Y)$ in $(\ZZ/p^m\ZZ)[X,Y]$. We mention that there is an alternative way for proving the property without using \eqref{Lang-text} as follows, and this is suggested by an anonymous reviewer. It is sufficient to show that the coefficient of $Y^{(p-1)/2}$ in $\cA_p(X,Y)$ is not divisible by $p$. The congruence subgroup $\Gamma_0(4)$ has three cusps $\{\infty, -\frac{1}{2}, 0\}$, and the constant terms in Fourier expansions of modular forms at cusps are given as follows: 
$$
\left\{
\begin{array}{l}
E_{p-1} \text{ has constant term } 1 \text{ at all the cusps},\\
\Theta \text{ has constant terms } \,\,1,\,\,\,0,\,\,\,\frac{1-i}{2} \text{ at the cusps } \infty,\,-\frac{1}{2},\,0, \text{ respectively},\\
F_2 \text{ has constant terms } 0,\,\frac{1}{16},\,-\frac{1}{64} \text{ at the cusps } \infty,\,-\frac{1}{2},\,0, \text{ respectively}
\end{array}
\right.
$$
(cf. \cite[Section 2 and p. 169]{Tup06}). Thus, the coefficient of $Y^{(p-1)/2}$ in $\cA_p(X,Y)$ is the constant term in the Fourier expansion of $E_{p-1}/F_2^{(p-1)/2}$ at $-\frac{1}{2}$, 
which is a power of 2.

(b) In Theorem \ref{Kob-Lan} we require the condition that $f\not\equiv 0$ $(\text{mod }p\ZZ_{(p)}[[q]])$. 
We point out that this condition is missing in \cite[Lemma 2.4]{LS14}; there is a counter example for this as follows. Let $k=p^{m-1}(p-1)+4$ and $f=pE_4E_{p-1}^{p^{m-1}}\in M_{k}(\ZZ_{(p)},\Gamma(1))$. Then we see that $w_{p^{m+1}}(\overline{f})=w_{p^{m+1}}(\overline{E}_4)<k$; however, $\overline{A}_p(X,Y)^{p^{m}}$ cannot divide $\overline{G}(f;X,Y)$ in $(\ZZ/p^{m+1}\ZZ)[X,Y]$. In fact, \cite[Proposition 4.3]{LS14} is proved by using \cite[Lemma 2.4]{LS14}, and it needs to be revised; for instance, we refer to Lemma \ref{LS-prop-4.3}, which is parallel to \cite[Proposition 4.3]{LS14}.}
\end{rmk}
\smallskip

\section{A quasi-valuation $\nu_p$ on $\ZZ_{(p)}[\Theta, F_2,E_2]$}\label{sec-4}

For a prime $p\ge5$, Larson and Smith \cite{LS14} defined a quasi-valuation $v_{p}$ 
on  the ring of $\Gamma(1)$-quasimodular forms 
$\ZZ_{(p)}[E_4,E_6,E_2]$, and they studied its properties 
by calculating the Rankin-Cohen bracket.
In this section, we define a quasi-valuation $\nu_p$ 
on the ring of $\Gamma_1(4)$-quasimodular forms 
$\ZZ_{(p)}[\Theta,F_2,E_2]$, and we prove that its properties are parallel to the  results of \cite{LS14}. 
We remark that using a formula by Zagier (see Proposition \ref{Zag-prop}) instead of the Rankin-Cohen bracket simplifies proofs.

We begin by introducing {\it quasi-valuations}.
\begin{defn}{\rm \cite[p. 319]{Sar12}}
Let $R$ be a commutative ring. A quasi-valuation on $R$ is a map $\nu:R\to \ZZ\cup\{\infty\}$ such that for all $x$, $y\in R$, 
\begin{enumerate}[(a)]
\item $\nu(x)=\infty$ if and only if $x=0$, 
\item $\nu(xy)\ge \nu(x)+\nu(y)$,
\item $\nu(x+y)\ge \min\{\nu(x),\nu(y)\}$.
\end{enumerate}
\end{defn}

We remark that $\nu$ is called a {\it valuation} on $R$ 
if it further satisfies $\nu(xy)= \nu(x)+\nu(y)$ for all $x$, $y\in R$.
Now, let $p\ge5$ be a prime.
By $\langle\cA_p^p,p\rangle$ we denote the ideal generated by $\cA_p^p$ and $p$ in the polynomial ring $\ZZ_{(p)}[X,Y,Z]$.
We define a map $\nu_p:\ZZ_{(p)}[X,Y,Z]\to \ZZ\cup\{\infty\}$ by 
$$\nu_p(\cG):=\sup\{n\mid \cG\in \langle\cA_p^p,p\rangle^n\}.$$
Then $\nu_p$ is a quasi-valuation. 
For a $\Gamma_1(4)$-quasimodular form $g\in\ZZ_{(p)}[\Theta,F_2,E_2]$, we simply write $\nu_p(g)$ instead of $\nu_p(\cG(g;X,Y,Z))$.
We also have that for all $g\in\ZZ_{(p)}[\Theta,F_2,E_2]$ 
$$\nu_p(Dg)\ge \nu_p(g) \,\,\text{ and }\,\, \nu_p(g_0)\ge \nu_p(g),$$
where $g_0$ is the modular part of $g$.

Let $k\in\ZZ$ or $\frac{1}{2}+\ZZ$ and 
$\Gamma'\subset \Gamma(1)$ be a congruence subgroup such that $\Gamma'\subset \Gamma_1(4)$ if $k\in\frac{1}{2}+\ZZ$. 
For $f\in M_k(\CC,\Gamma')$, we define a sequence of modular forms $f_n\in M_{k+2n}(\CC,\Gamma')$ recursively by
$$f_{n+1}:=\left(D f_n-\frac{k+2n}{12}E_2f_n\right)-\frac{n(n+k-1)}{144}E_4f_{n-1}\quad (n\ge 0)$$
with initial condition $f_0=f$ and $f_{-1}=0$.
Then a formula by Zagier \cite[(37)]{Zag94} is equivalent to the following proposition (we also refer to \cite[p. 55]{Zag08}).
\begin{prop}\label{Zag-prop}
With the same notation as above, we have for any $n\ge 0$ 
$$D^nf=\sum_{j=0}^{n}
{{n}\choose{j}} {{n+k-1}\brack{n+k-1-j}}
f_{n-j} \big(\frac{E_2}{12}\big)^j,$$
where for $x>y\ge 0$, 
${{x}\brack{y}}:=x(x-1)(x-2)\cdots (y+2)(y+1)$ and ${{y}\brack{y}}:=1$.
\end{prop}
Now, we further suppose that $f\in M_k(\ZZ_{(p)},\Gamma_1(4))$.
By Lemma \ref{p-integral-coefficients} and Proposition \ref{Zag-prop}, we have that 
\begin{equation}\label{Zag-eq-1}
\cG\big(\big(D^{p}f-(D^{p}f)_0\big);X,Y,Z\big)\in p\ZZ_{(p)}[X,Y,Z]
\end{equation}
and
\begin{equation}\label{Zag-eq-2}
\cG\big(\big(D^{p^2}f-(D^{p^2}f)_0\big);X,Y,Z\big)\in p^2\ZZ_{(p)}[X,Y,Z].
\end{equation}

\begin{lem}\label{LS-prop-4.3}
Let $p\ge 5$ be a prime and $k\in\ZZ$ or $\frac{1}{2}+\ZZ$.
For $f\in M_k(\ZZ_{(p)},\Gamma_1(4))$, we have
$\nu_p(D^{p^2}f)\ge 2$.
\end{lem}
\begin{pf}
By \eqref{Zag-eq-2} and applying Euler's totient theorem to Fourier coefficients, we have 
$$
(D^{p^2}f)_0 \equiv D^{p^2}f \equiv D^pf \quad 
(\text{mod } p^2\ZZ_{(p)}[[q]]).
$$
By \eqref{Zag-eq-1} we write 
$$(D^pf)_0=D^pf-pH(E_2),$$ 
where $H(E_2)=\sum_{j=1}^{p} h_{(j)} E_2^j$ and 
$h_{(j)}:=\frac{1}{p}{{p}\choose{j}} {{p+k-1}\brack{p+k-1-j}}\frac{f_{p-j}}{12^j}\in M_{k+2(p-j)}(\ZZ_{(p)},\Gamma_1(4))$. 
Therefore, we have 
$$(D^pf)_0 \equiv
(D^{p^2}f)_0-p\sum_{j=1}^{p} h_{(j)} E_{p-1}^{2p-j} 
E_{p+1}^j \quad (\text{mod } p^2\ZZ_{(p)}[[q]]),$$
using the fact that 
$E_{p-1}\equiv 1$ $(\text{mod } p\ZZ_{(p)}[[q]])$ and 
$E_{p+1}\equiv E_2$ $(\text{mod } p\ZZ_{(p)}[[q]])$.
Let $$g:=(D^{p^2}f)_0-p\sum_{j=1}^{p} h_{(j)} E_{p-1}^{2p-j} 
E_{p+1}^j.$$
We note that  
$g\in M_{k+2p^2}(\ZZ_{(p)},\Gamma_1(4))$, 
$(D^{p}f)_0\in M_{k+2p}(\ZZ_{(p)},\Gamma_1(4))$ 
and their reductions satisfy
$\overline{g}=\overline{(D^{p}f)_0}$ in $(\ZZ/p^2\ZZ)[[q]]$; so, we have $w_{p^2}(\overline{g})\le (k+2p^2)-2p(p-1)$.
Now, we claim that 
\begin{equation}\label{Kob-Lan-eq-1}
\nu_p(g)\ge 2. 
\end{equation}
We first assume that $g\not\equiv 0$ $(\text{mod }p\ZZ_{(p)}[[q]])$. 
By applying Theorem \ref{Kob-Lan} to $\overline{g}$, 
there exists an isobaric element $F(X,Y)$ of $\ZZ_{(p)}[X,Y]$ such that $\overline{\cG}(g;X,Y)=
\overline{F}(X,Y)\overline{\cA}_p(X,Y)^p$ in $(\ZZ/p^2\ZZ)[X,Y]$. Let $g':=F(\Theta,F_2)\in M_{k+p^2+p}(\ZZ_{(p)},\Gamma_1(4))$.
Then we have $\overline{g'}=\overline{g}=\overline{(D^{p}f)_0} $ in $(\ZZ/p^2\ZZ)[[q]]$, $g'\not\equiv 0$ $(\text{mod }p\ZZ_{(p)}[[q]])$ and $w_{p^2}(\,\overline{g'}\,)\le (k+p^2+p)-p(p-1)$. 
By applying Theorem \ref{Kob-Lan} to $\overline{g'}$, 
$\overline{\cA}_p(X,Y)^{p}$ divides $\overline{F}(X,Y)$ 
in $(\ZZ/p^2\ZZ)[X,Y]$. Therefore, we have that 
$\overline{\cA}_p(X,Y)^{2p}\mid\overline{\cG}(g;X,Y)$ in $(\ZZ/p^2\ZZ)[X,Y]$; thus, we have $\nu_p(g)\ge2$.
Secondly, we assume that $g\equiv 0$ $(\text{mod }p\ZZ_{(p)}[[q]])$ and $g\not\equiv 0$ $(\text{mod }p^2\ZZ_{(p)}[[q]])$. Let $u_{(0)}:=\frac{1}{p}g$. 
Then $u_{(0)}\in M_{k+2p^2}(\ZZ_{(p)},\Gamma_1(4))$, $\frac{1}{p}(D^{p}f)_0\in M_{k+2p}(\ZZ_{(p)},\Gamma_1(4))$, 
$\overline{u}_{(0)}=\overline{\tfrac{1}{p}(D^{p}f)_0}\neq 0$ in $(\ZZ/p\ZZ)[[q]]$ and $w_p(\overline{u}_{(0)})
=w_p\big(\overline{\tfrac{1}{p}(D^{p}f)_0}\big)\le (k+2p^2)-2p(p-1)$. 
For $1\le j\le p$, we define $u_{(j)}$ recursively as follows. 
By applying Theorem \ref{Kob-Lan} to $\overline{u}_{(j-1)}$, there exists an isobaric element $F_{(j)}(X,Y)$ of $\ZZ_{(p)}[X,Y]$ such that $\overline{\cG}\big(u_{(j-1)};X,Y\big)=
\overline{F}_{(j)}(X,Y)\overline{\cA}_p(X,Y)$ in $(\ZZ/p\ZZ)[X,Y]$. Let $u_{(j)}:=F_{(j)}(\Theta,F_2)\in M_{k+2p^2-j(p-1)}(\ZZ_{(p)},\Gamma_1(4))$. 
Then $\overline{u}_{(j)}=\overline{u}_{(j-1)}=\overline{\tfrac{1}{p}(D^{p}f)_0}\neq 0$ in $(\ZZ/p\ZZ)[[q]]$ and 
$w_p(\overline{u}_{(j)})\le \big(k+2p^2-j(p-1)\big)-(2p-j)(p-1)$. 
Hence, we have that
$\overline{\cG}\big(u_{(0)};X,Y\big)
=\overline{\cG}\big(u_{(1)};X,Y\big)\overline{\cA}_p(X,Y)
=\cdots
=\overline{\cG}\big(u_{(p)};X,Y\big)\overline{\cA}_p(X,Y)^p
$ in $(\ZZ/p\ZZ)[X,Y]$; so, $\nu_p(g)\ge \nu_p(u_{(0)})+1\ge 2$. Finally, we assume that $g\equiv 0$ $(\text{mod }p^2\ZZ_{(p)}[[q]])$; then it is clear that $\nu_p(g)\ge 2$.
Therefore, the claim \eqref{Kob-Lan-eq-1} follows.
Consequently, by \eqref{Zag-eq-2} and \eqref{Kob-Lan-eq-1}, we get 
\vspace{10pt}

$\displaystyle\,\,
\nu_p(D^{p^2}f) 
\ge {\min\left\{\nu_p(D^{p^2}f-(D^{p^2}f)_0), 
\,\,\nu_p(g),\,\, 
\nu_p\bigg(p\sum_{j=1}^{p} h_{(j)} E_{p-1}^{2p-j} E_{p+1}^j\bigg)\right\}} 
\ge 2.$
\end{pf}
\begin{lem}\label{LS-prop-4.4}
Let $p\ge 5$ be a prime. Then 
\begin{enumerate}[(a)]
\item\label{LS-3-2} $\nu_p(D^{2}E_{p-1})\ge 1$,
\item\label{LS-4-4} $\nu_p(D^{p^2}E_{p-1}^p)\ge 3$.
\end{enumerate}
\end{lem}
\begin{pf}
\eqref{LS-3-2}
By Proposition \ref{Zag-prop} we have
$$D^2E_{p-1}=(E_{p-1})_2+2p(E_{p-1})_1 \frac{E_2}{12} +p(p-1) E_{p-1} \left(\frac{E_2}{12}\right)^2.$$
We note that
$(E_{p-1})_1=DE_{p-1}-\frac{p-1}{12}E_2 E_{p-1}$ and $(E_{p-1})_2=D(E_{p-1})_1-\frac{p+1}{12}E_2(E_{p-1})_1-\frac{p-1}{144}E_4E_{p-1}$.
We thus have that
$$
\begin{array}{rll}
(E_{p-1})_1-\frac{1}{12}E_{p+1}
&\in& pM_{p+1}(\ZZ_{(p)},\Gamma_1(4)),\\
(E_{p-1})_2
&\in & pM_{p+3}(\ZZ_{(p)},\Gamma_1(4))
\end{array}
$$
since $E_{p-1}\equiv 1$ $\big(\text{mod }p\ZZ_{(p)}[[q]]\big)$, $E_{p+1}\equiv E_2$ $\big(\text{mod }p\ZZ_{(p)}[[q]]\big)$, and $DE_2=(E_2^2-E_4)/12$.
By Lemma \ref{p-integral-coefficients} we have that 
$\cG\big((E_{p-1})_1;X,Y\big)\in\ZZ_{(p)}[X,Y]$ and
$\cG\big((E_{p-1})_2;X,Y\big)\in p\ZZ_{(p)}[X,Y]$; so, we get
$\cG(D^2E_{p-1};X,Y,Z)\in p\ZZ_{(p)}[X,Y,Z]$.\\
\noindent
\eqref{LS-4-4}
By the product rule, we have
$$D^{p^2}E_{p-1}^p=\sum_{j_1+\cdots+j_p=p^2} \frac{p^2!}{j_1!\cdots j_p!}(D^{j_1}E_{p-1})\cdots (D^{j_p}E_{p-1}).$$
\noindent
If $p\nmid j_r$ for some $r$, then
$p^2 \mid \frac{p^2!}{j_1!\cdots j_p!}$ and there exists $s$ such that $j_s\ge 2$; so, by \eqref{LS-3-2} we have that 
$\nu_p\big(\frac{p^2!}{j_1!\cdots j_p!}(D^{j_1}E_{p-1})\cdots (D^{j_p}E_{p-1})\big)\ge 3$.
If $p\mid j_r$ and $j_r\neq p^2$ for every  $r$, then $p \mid \frac{p^2!}{j_1!\cdots j_p!}$ and there exist two different $s_1$, $s_2$ such that $j_{s_1}\ge p$, $j_{s_2}\ge p$; thus, by \eqref{LS-3-2} we see that
$\nu_p\big(\frac{p^2!}{j_1!\cdots j_p!}(D^{j_1}E_{p-1})\cdots (D^{j_p}E_{p-1})\big)\ge 3$.
The set of the remaining ordered pairs is $\cC:=\bigcup_{r=1}^{p}\{(j_1,\cdots,j_p) \mid j_r=p^2,\,j_s=0\,\text{ for all }\,s\neq r\}$. By Lemma \ref{LS-prop-4.3} we have that
$\nu_p\big(\sum_{\cC}\frac{p^2!}{j_1!\cdots j_p!}(D^{j_1}E_{p-1})\cdots (D^{j_p}E_{p-1})\big)=\nu_p(p E_{p-1}^{p-1}D^{p^2}E_{p-1})\ge 3$. 
Therefore, we get $\nu_p(D^{p^2}E_{p-1}^p)\ge 3$.
\end{pf}

\begin{lem}\label{LS-prop-4.5}
Let $p\ge 5$ be a prime and $k\in\ZZ$ or $\frac{1}{2}+\ZZ$.
For $f\in M_k(\ZZ_{(p)},\Gamma_1(4))$, we have
$\nu_p(D^{p^2}f)\ge \max\{\nu_p(f)+1,2\}$.
\end{lem}
\begin{pf}
We have 
$f=\sum_{r=0}^{\nu_p(f)}p^{\nu_p(f)-r}h_{(r)} E_{p-1}^{rp}$, where 
$h_{(r)}\in M_{k-rp(p-1)}(\ZZ_{(p)},\Gamma_1(4))$.
By the product rule, we have that
\begin{eqnarray*}
D^{p^2}(h_{(r)}E_{p-1}^{rp})
&=&\sum_{j_0=0}^{p^2} {{p^2}\choose{j_0}}(D^{j_0}h_{(r)})(D^{p^2-j_0}E_{p-1}^{rp})\\
&=&\sum_{j_0=0}^{p^2} {{p^2}\choose{j_0}}(D^{j_0}h_{(r)})
\sum_{j_1+\cdots+j_r=p^2-j_0}\frac{(p^2-j_0)!}{j_1!\cdots j_r!}(D^{j_1}E_{p-1}^p)\cdots (D^{j_r}E_{p-1}^p).
\end{eqnarray*}
If $j_0=0$ and $j_s\neq p^2$ for all $1\le s \le r$, then $p\mid \frac{p^2!}{j_1!\cdots j_r!}$; 
thus, we have 
$$\nu_p\bigg(h_{(r)}\sum_{\substack{j_1+\cdots+j_r=p^2\\p^2\nmid j_s}}\frac{p^2!}{j_1!\cdots j_r!}(D^{j_1}E_{p-1}^p)\cdots (D^{j_r}E_{p-1}^p)\bigg)\ge r+1$$
since $\nu_p(D^{j_s}E_{p-1}^p)\ge \nu_p(E_{p-1}^p)=1$.
If $j_0=0$ and $j_s=p^2$ for some $1\le s\le r$, then $\nu_p\big(h_{(r)}\sum_{s=1}^{r} E_{p-1}^{(r-1)p}
(D^{p^2}E_{p-1}^p)\big)\ge r+2$ by Lemma \ref{LS-prop-4.4} \eqref{LS-4-4}.
If $1\le j_0<p^2$, then $p\mid {{p^2}\choose{j_0}}$; 
hence we have $\nu_p\big({{p^2}\choose{j_0}}(D^{j_0}h_{(r)})(D^{p^2-j_0}E_{p-1}^{rp})\big)\ge r+1$ as $\nu_p(D^{p^2-j_0}E_{p-1}^{rp})\ge \nu_p(E_{p-1}^{rp})=r$.
If $j_0=p^2$, then $\nu_p\big((D^{p^2}h_{(r)})E_{p-1}^{rp}\big)\ge r+2$ by Lemma \ref{LS-prop-4.3}.
Therefore, we have $\nu_p\big(D^{p^2}(h_{(r)}E_{p-1}^{rp})\big)\ge r+1$; hence by Lemma \ref{LS-prop-4.3} we get $\nu_p(D^{p^2}f)\ge \max\{\nu_p(f)+1,2\}$.
\end{pf}

\begin{prop}\label{LS-lem-4.8} 
Let $p\ge 5$ be a prime and $k\in\ZZ$ or $\frac{1}{2}+\ZZ$.
For $f\in M_k(\ZZ_{(p)},\Gamma_1(4))$ and 
$m\in\NN$, we have 
$\nu_p(D^{mp^2}f)\ge m+1$.
\end{prop}
\begin{proof} 
We prove that $\nu(D^{mp^2}f)\ge m+1$ by induction on $m$. If $m=1$, we have $\nu_p(D^{p^2}f)\ge 2$ by Lemma \ref{LS-prop-4.3}. 
Suppose that for all $m\le n$, $\nu_p(D^{mp^2}f)\ge m+1$.
By Proposition \ref{Zag-prop} we have
$$\displaystyle D^{np^2}f=\sum_{j=0}^{np^2}a_j f_{np^2-j}\big(\frac{E_2}{12}\big)^j,$$
where $a_j:={{np^2}\choose{j}} {{np^2+k-1}\brack{np^2+k-1-j}}$. 
Then we have that
$$\nu_p\big(D^{(n+1)p^2}f\big)
\ge \min_{0\le j\le np^2}
\left\{ 
\nu_p\left(D^{p^2}\left(a_{j}f_{np^2-j} \left(\frac{E_2}{12}\right)^j\right) \right) \right\}
\ge \min_{0\le j\le np^2} \{c_j\},
$$
where 
$c_0:=\nu_p\big(D^{p^2} f_{np^2}\big)$ and 
$c_j:=\nu_p\big(a_{j}f_{np^2-j} \big)$ for $1\le j \le np^2$.

By Lemma \ref{LS-prop-4.5}
it follows that
$\nu_p(f_{np^2}) \ge  \nu_p (D^{np^2}f)\ge n+1$ and $\nu_p(D^{p^2}f_{np^2})\ge \nu_p(f_{np^2})+1$; so, we get $c_0\ge n+2$.
If $1\le j \le p^2$, then $\nu_p(f_{np^2-j})\ge \nu_p(D^{np^2-j}f)\ge 
\nu_p(D^{(n-1)p^2}f)\ge n$. 
Also, we have that $p^2\mid {{np^2}\choose{j}}$ for $1\le j <p$, and $p$ divides both ${{np^2}\choose{j}}$ 
and ${{np^2+k-1}\brack{np^2+k-1-j}}$ for $p\le j <p^2$; thus, we have $c_j\ge n+2$ for $1\le j<p^2$.
If $p^2\le j \le np^2$, then $p^{p\lfloor j/p^2\rfloor}\mid {{np^2+k-1}\brack{np^2+k-1-j}}$ and 
$\nu_p(f_{np^2-j})\ge \nu_p(D^{(n-\lceil j/p^2\rceil)p^2}f)\ge n-\lceil j/p^2\rceil+1$; hence we obtain 
$c_j\ge n+2$ for $p^2\le j \le np^2$.
Therefore, we get $\nu_p(D^{(n+1)p^2}f)\ge n+2$.
\end{proof}
\smallskip

\section{Proofs of Theorems \ref{main-thm-0} and \ref{cor-thm-0}}\label{sec-5}
In this section we prove Theorem \ref{main-thm-0} and Theorem \ref{cor-thm-0}.
For $z=x+iy\in\CC$, we denote 
$$E_2^\ast(z):=E_2(z)-\frac{3}{\pi y}.$$

\begin{lem}\label{LS-lem-2.7}
Let $k\in\ZZ$ or $\frac{1}{2}+\ZZ$ and $f\in M_k(\CC,\Gamma_1(4))$.
Then $\partial^n f=\cG(D^nf ; \Theta,F_2,E_2^\ast)$.
\end{lem}
\begin{pf} 
Let $\phi:\CC[\Theta,F_2,E_2]\to 
\CC[\Theta,F_2,E_2^\ast]$ be the map that sends $E_2$ to $E_2^\ast$. 
It follows that $\partial \circ \phi=\phi \circ D$
from \eqref{D-basis} and
\begin{equation}\label{Theta-E_2}
\left\{
\begin{array}{l}
\partial \Theta =
(\Theta E_2^\ast  -\Theta^5 +80\Theta F_2 )/24,\\
\partial F_2 = 
(F_2 E_2^\ast +5\Theta^4 F_2-16F_2^2)/6,\\
\partial E_2^\ast=
({E_2^\ast}^2-\Theta^8-224\Theta^4F_2-256F_2^2)/12.
\end{array}
\right.
\end{equation}
We show $\partial^n f=\phi\circ\cG(D^nf ; \Theta,F_2,E_2)$ by induction on $n$.
If $n=0$, it is clear. 
Now, suppose that $\partial^n f= \phi\circ D^n f$. Then we have 
$\partial^{n+1}f=\partial\circ\partial^{n}f=\partial\circ \phi\circ D^n f=\phi\circ D^{n+1}f$.
\end{pf}

Let $p\ge5$ be a prime, $k\in \ZZ$ or $\frac{1}{2}+\ZZ$ and $f\in M_k({\ZZ},\Gamma_1(4))$. 
Let $K$ be a quadratic field of discriminant $d<0$,
let $\tau_0$ be a CM point in $K$, and let
$$c_n(f,\tau_0,\Omega):=\frac{\partial^n f(\tau_0)}{\Omega^{2n+k}}.$$
We note that ${\Theta(\tau_0)}/{\Omega_K^{1/2}},{F_2(\tau_0)}/{\Omega_K^2},{E_2^{\ast}(\tau_0)}/{\Omega_K^2}\in\overline{\QQ}$.
Let $F$ be a finite Galois extension of $\QQ$ containing these algebraic numbers, and let $\fp$ be a prime of $F$ lying above $p$. 
There exists the ramification index $e(p,F/\QQ)\in\NN$ such that $\ord_{\fp}\vert_{\QQ}=e(p,F/\QQ)\,\ord_p$, where $\ord_{\fp}:F\to\ZZ$ and $\ord_p:\QQ\to \ZZ$ are the usual discrete valuations. 
We define $\ord_p:F\to \QQ$ to be 
\begin{equation}\label{p-ord}
\ord_p:=\frac{1}{e(p,F/\QQ)}\,\ord_{\fp}.
\end{equation}
By removing denominators of these three algebraic numbers, we choose $\Omega_{\tau_0}$ to be an algebraic multiple of $\Omega_K$ such that for all primes $p\ge5$, 
$$
\ord_{p}\big({\Theta(\tau_0)}/{\Omega_{\tau_0}^{1/2}}\big)\ge 0,
\quad
\ord_{p}\big({F_2(\tau_0)}/{\Omega_{\tau_0}^2}\big)\ge 0,
\quad
\ord_{p}\big({E_2^{\ast}(\tau_0)}/{\Omega_{\tau_0}^2}\big)\ge 0.
$$
Then $\Omega_{\tau_0}$ satisfies \eqref{eq-trascendental} 
since $c_n(f,\tau_0,\Omega_{\tau_0})\in \ZZ_{(6)}[\Theta(\tau_0)/\Omega_{\tau_0}^{1/2},F_2(\tau_0)/\Omega_{\tau_0}^{2},E_2^\ast(\tau_0)/\Omega_{\tau_0}^{2}]$.

\begin{lem}\label{LS-lem-5.1}
With the same notation as above, 
if $p\ge 5$ is a prime such that $(\frac{d}{p})\in\{0,-1\}$, then we have that $c_n(E_{p-1},\tau_0,\Omega_{\tau_0})\equiv 0$ $(\text{mod }p)$. 
\end{lem}
\begin{pf}
See \cite[Lemma 5.1]{LS14}.
\end{pf}
\smallskip 

\noindent
{\bf Proof of Theorem \ref{main-thm-0}}\,\,\,
Let $n\ge (m-1) p^2$ with $m>1$.
By Proposition \ref{LS-lem-4.8} we have
$$\cG(D^nf ; X,Y,Z)=\sum_{0\le j \le m}p^{m-j}\cA_p^{jp}\cG_j,$$
where $\cG_j\in \ZZ_{(p)}[X,Y,Z]$. 
By Lemma \ref{LS-lem-2.7} we have
\begin{eqnarray*}
c_n(f,\tau_0,\Omega_{\tau_0}) 
&=&\frac{\cG\big(D^nf ; \Theta(\tau_0),F_2(\tau_0),E_2^\ast(\tau_0)\big)}{\Omega_{\tau_0}^{2n+k}}\\
&=&\sum_{0\le j \le m}p^{m-j}
c_n(E_{p-1},\tau_0,\Omega_{\tau_0})^{jp}
\frac{\cG_j\big(\Theta(\tau_0),F_2(\tau_0),E_2^\ast(\tau_0)\big)}{\Omega_{\tau_0}^{2n+k-jp(p-1)}}.
\end{eqnarray*}
By Lemma \ref{LS-lem-5.1} we have $c_n(f,\tau_0,\Omega_{\tau_0})\equiv 0$ $(\text{mod }p^m)$.
\qed\\
\smallskip

\noindent
{\bf Proof of Theorem \ref{cor-thm-0}}\,\,\,
The Jacobi theta functions are defined as 
\begin{eqnarray*}
\theta_2(z)=\sum_{n=-\infty}^{\infty}e^{\pi i (n+1/2)^2 z},
\quad
\theta_3(z):=\sum_{n=-\infty}^{\infty}e^{\pi i n^2 z},
\quad
\theta_4(z):=\sum_{n=-\infty}^{\infty}(-1)^n e^{\pi i n^2 z}.
\end{eqnarray*}
Let $\displaystyle a:=\frac{\gamma(1/4)}{2^{1/2}\pi^{3/4}}$.
It is well-known that 
\begin{eqnarray*}
\theta_2(i/2)=2^{3/8}a,
\qquad
\theta_3(i/2)=\frac{(2^{1/2}+1)^{1/2}}{2^{1/4}}a 
\qquad\text{and}\qquad
\Theta(i/2)=\theta_3(i)=a.
\end{eqnarray*}
(We refer to \cite[pp. 325]{Ber98}. We note that $\theta_2(z)=2e^{\frac{\pi i z}{4}}\psi(e^{2\pi iz})$ and 
$\theta_3(z)=\varphi(e^{\pi iz})$, where $\psi$ and $\varphi$ are defined in \cite[pp. 323]{Ber98}.)
Since $\theta_4^4=\theta_3^4-\theta_2^4$,
we have that 
$$\theta_4(i/2)=\frac{(2^{1/2}-1)^{1/2}}{2^{1/4}}a.$$
Furthermore, we have that 
$$E_4(i/2)=\frac{33}{4}a^8$$
as $E_4(z)=\frac{1}{2}\big(\theta_2(z)^8+\theta_3(z)^8+\theta_4(z)^8\big)$ (see \cite[pp. 28-29]{Zag08}).
By \eqref{E4-F2} we have that 
$$F_2(i/2)=\frac{1}{32}a^4.$$
From \eqref{Theta-E_2} and $\partial \Theta(i/2)=0$, we have that
$$E_2^\ast(i/2)=-\frac{3}{2}a^4.$$
Let $K:=\QQ(\sqrt{-4})$ and let $\tau_0:=iy_0:=i/2\in K$.
We have $\Omega_K=a^2/2$ and take 
$$\Omega_{\tau_0}:=\frac{1}{2^{1/4}}\Omega_K=\frac{1}{2^{5/4}}a^2.$$
Then $\Omega_{\tau_0}$ satisfies \eqref{eq-trascendental} since
$$
\frac{\Theta(\tau_0)}{\Omega_{\tau_0}^{1/2}}=2^{5/8},
\quad
\frac{F_2(\tau_0)}{\Omega_{\tau_0}^2}=2^{-5/2},
\quad
\frac{E_2^{\ast}(\tau_0)}{\Omega_{\tau_0}^2}=-2^{3/2}\cdot3.
$$
Now, we consider the Taylor series of $\Theta(z)$ around $w=\tau_0$:
\begin{eqnarray}\label{modified-Romik-seq}
(1-w)^{-1/2}\Theta\Big(\frac{\tau_0-\overline{\tau}_0w}{1-w}\Big)
&=&(1-w)^{-1/2}\Theta\Big(\tau_0+\frac{2iy_0w}{1-w}\Big)\nonumber\\
&=&(1-w)^{-1/2}\sum_{r=0}^{\infty}\frac{D^r \Theta(\tau_0)}{r!} \Big(\frac{-4\pi y_0w}{1-w}\Big)^r\nonumber\\
&=&\sum_{n=0}^{\infty}\partial^n\Theta(\tau_0) \frac{(-4\pi y_0 w)^n}{n!}.
\end{eqnarray}
By comparing \eqref{Romik-seq} with \eqref{modified-Romik-seq}, we have that 
$$d(n)=2^{-5/8}\cdot
\frac{\partial^{2n}\Theta(\tau_0)}{\Omega_{\tau_0}^{4n+{1}/{2}}}$$
as $\Theta(z)=\theta_3(2z)$.
By Theorem \ref{main-thm-0} we have that 
$$d(n)\equiv 0 \quad (\text{mod }p^m),$$
where $p\ge5$ is a prime such that $p\equiv 3$ $(\text{mod }4)$, $m\ge 2$, and $n\ge\lceil (m-1)p^2/2\rceil$.
\qed

\smallskip

\section*{Acknowledgement}
The authors were supported by the Basic Science Research Program
through the National Research Foundation of Korea (NRF) funded
by the Ministry of Education (Grant No. 2019R1A6A1A11051177).
J. Kim was also supported
by the Basic Science Research Program
through the National Research Foundation of Korea (NRF) funded
by the Ministry of Education (Grant No. 2020R1I1A1A01074746),
and Y. Lee was also supported
by the National Research Foundation of Korea (NRF) grant funded
by the Korea government (MEST)(NRF-2022R1A2C1003203). 
The authors thank the reviewers for the valuable suggestions.

\bigskip

\noindent Institute of Mathematical Sciences,

\noindent Ewha Womans University,

\noindent Seoul, Republic of Korea,

\noindent E-mail: jigu.kim@ewha.ac.kr\bigskip

\noindent Department of Mathematics, Ewha Womans University,

\noindent and Korea Institute for Advanced Study

\noindent Seoul, Republic of Korea,

\noindent E-mail: yoonjinl@ewha.ac.kr


\smallskip

\begin{thebibliography}{plain} %plain %alpha

\bibitem{Ber98} 
B. C. Berndt, 
\emph{Ramanujan's notebook Part V}, 
Springer-Verlag, New York, 1998.

\bibitem{GMR20} 
P. Guerzhoy, M. H. Mertens, and L. Rolen,
\emph{Periodicities for Taylor coefficients of half-integral weight modular forms},
Pacific J. Math. {\bf 307} (2020), 137--157.

\bibitem{Kat73}
N. M. Katz,
\emph{$p$-adic properties of modular schemes and modular forms},
in: \emph{Modular functions of one variable III (Proc. Internat. Summer School, Univ. Antwerp, Antwerp, 1972)}, 69--190, 
Lecture Notes in Math. {\bf 350}, Springer, Berlin, 1973.

\bibitem{Kob86}
N. Koblitz,
\emph{$p$-adic congruences and modular forms of half integer weight}, 
Math. Ann. {\bf 274} (1986), 199--229.

\bibitem{Kob93}
------, %N. Koblitz,
\emph{Introduction to elliptic curves and modular forms}, 
Second edition,
Graduate Texts in Mathematics {\bf 97}, 
Springer-Verlag, New York, 1993.

\bibitem{Lan95}
S. Lang,
\emph{Introduction to modular forms}, 
Grundlehren der Mathematischen Wissenschaften {\bf 222}, 
Springer-Verlag, Berlin, 1995.
%With appendixes by D. Zagier and W. Feit, Corrected reprint of the 1976 original. 

\bibitem{LS14} 
H. Larson and G. Smith,
\emph{Congruence properties of Taylor coefficients of modular forms},
Int. J. Number Theory {\bf 10} (2014), 1501--1518.

\bibitem{Rom20}
D. Romik, 
\emph{The Taylor coefficients of the Jacobi theta constnat $\theta_3$},
Ramanujan J. {\bf 52} (2020), 275--290.

\bibitem{Sar12}
S. Sarussi, 
\emph{Quasi-valuations extending a valuation}, 
J. Algebra {\bf 372} (2012), 318--364.

\bibitem{Sch21}
R. Scherer, 
\emph{Congruences modulo primes of the Romik sequence related to the Taylor expansion of the Jacobi theta constant $\theta_3$},
Ramanujan J. {\bf 54} (2021), 427--448.

\bibitem{Ser73}
J. P. Serre,
\emph{Formes modulaires et fonctions zêta $p$-adiques},
in: \emph{Modular functions of one variable III (Proc. Internat. Summer School, Univ. Antwerp. Antwerp, 1972)}, 191--268, 
Lecture Notes in Math. {\bf 350}, Springer, Berlin, 1973.

\bibitem{Shi73}
G. Shimura, 
\emph{Modular forms of half integral weight}, 
Ann. Math. {\bf 97} (1973), 440--481.

\bibitem{Tup06}
A. Tupan, 
\emph{Congruences for $\Gamma_1(4)$-modular forms of half-integral weight}, 
Ramanujan J. {\bf 11} (2006), 165--173.

\bibitem{VW14}
J. Voight and J. Willis, 
\emph{Computing power series expansions of modular forms}, 
in: \emph{Computations with modular forms}, Contrib. Math. Comput. Sci. {\bf 6}, Springer, 2014, 257--277.

\bibitem{Wak20}
T. Wakhare, 
\emph{Romik's conjecture for the Jacobi theta function}, 
J. Number Theory {\bf 215} (2020), 1--22.

\bibitem{WV20}
T. Wakhare and C. Vignat,
\emph{Taylor coefficients of the Jacobi $\theta_3(q)$ function}, 
J. Number Theory {\bf 216} (2020), 280--306.

\bibitem{Zag94} 
D. Zagier,
\emph{Modular forms and differential operators},
Proc. Indian Acad. Sci. (Math. Sci.) {\bf 104} 
(1994), 57--75.

\bibitem{Zag08}
------, %D. Zagier,
\emph{Elliptic modular forms and their applications}, 
in: \emph{The 1-2-3 of Modualr Forms}, 1--103, 
Universitext, Springer, Berlin, 2008.

\end{thebibliography}
\end{document}